\documentclass[11pt]{amsart}

\usepackage[T1]{fontenc}
\usepackage[utf8]{inputenc}
\usepackage{geometry}
\usepackage{amssymb}
\usepackage{amsmath}
\usepackage{mathtools}
\usepackage{xcolor}
\usepackage{enumitem}
\usepackage[colorlinks=true,citecolor=red,linkcolor=blue,urlcolor=red]{hyperref}

\DeclareMathOperator{\conv}{Conv}
\DeclareMathOperator{\lct}{lct}
\DeclareMathOperator{\barc}{bar}

\newcommand{\RR}{\mathbb{R}}
\newcommand{\ZZ}{\mathbb{Z}}
\newcommand{\QQ}{\mathbb{Q}}
\newcommand{\PP}{\mathbb{P}}

\newtheorem{thm}{Theorem}[section]
\newtheorem{lem}[thm]{Lemma}
\newtheorem{prop}[thm]{Proposition}
\newtheorem{cor}[thm]{Corollary}
\newtheorem{conj}[thm]{Conjecture}

\theoremstyle{definition}
\newtheorem{rem}[thm]{Remark}

\numberwithin{equation}{section}

\begin{document}
	
	\title[The alpha spectrum of K-polystable toric $\QQ$-Fano varieties]
	{The alpha spectrum of K-polystable toric $\QQ$-Fano varieties}
	
	\author{Xian Wu}
	\address{}
	\email{scalarscales@gmail.com}
	
	\keywords{alpha invariant, K-polystable, toric variety, Fano polytope}
	\date{}
	
	\begin{abstract}
		We explicitly determine the numerical spectrum of
		the ordinary, non-equivariant global invariant for $n$-dimensional
		K-polystable toric $\QQ$-Fano varieties.
		Specifically, we prove that
		$$
		\left\{
		\alpha(X)\mid
		X\text{ is an $n$-dimensional K-polystable toric $\QQ$-Fano variety}
		\right\}
		=
		\QQ\cap
		\left[\frac{1}{n+1},\frac{1}{2}\right].
		$$
		This result establishes a complete toric realization theorem, which strengthens a question raised by Liu and Zhuang and refines the recent construction results of Liu and Zhu.
		The initial construction of the examples in this work was suggested by GPT-5.6 Sol, and was subsequently refined and rigorously developed by the author.
	\end{abstract}
	
	\maketitle
	
	\section{Introduction}\label{sec:introduction}
	
	In this paper, all varieties are defined over the complex field $\mathbb C$. 
	For a $\QQ$-Fano variety $X$, its ordinary, non-equivariant global invariant is
	$$
	\alpha(X)=\inf\{\lct(X,D)\mid D\sim_{\QQ}-K_X,\ D\geq0\}.
	$$

	Fujita and Odaka proved that an $n$-dimensional K-semistable $\QQ$-Fano variety satisfies
	\begin{equation}\label{eq:FO-lower}
		\alpha(X)\geq\frac1{n+1};
	\end{equation}
	see \cite[Theorem~3.5]{FO18}.  
	The lower bound is attained by $\PP^n$. 
	In the smooth setting, Jiang proved that equality in \eqref{eq:FO-lower} characterizes $\PP^n$ and proposed the following
	gap conjecture.
	
	\begin{conj}[{\cite[Conjecture~1.6]{Jia17}}]\label{conj:Jiang}
		Let $X$ be an $n$-dimensional K-semistable Fano manifold.  
		Then
		$$
		\alpha(X)<\frac1n\quad\Longleftrightarrow\quad X\cong\PP^n.
		$$
	\end{conj}
	
	For general $\QQ$-Fano varieties with dimension $\geq 2$, Liu and Zhuang asked whether there
	exists a K-semistable example with
	$$
	\frac{1}{n+1}<\alpha(X)<\frac{1}{n}
	$$
	\cite[Question~1.6(1)]{LZ22}. 
	 
	Liu and Zhu answered this affirmatively:
	for every $n\geq2$, there exists a K-polystable toric $\QQ$-Fano variety with $\alpha(X)=2/(2n+1)$ \cite[Theorem~1.4]{LZ26}.  
	They then asked whether one can enter the smaller interval between $1/(n+1)$ and $2/(2n+1)$ \cite[Question~1.5]{LZ26}.\\
	
	My main result determines the full numerical spectrum.	
	\begin{thm}\label{thm:main}
		For every integer $n\geq1$, for any given rational number 
		$$
		r\in \mathbb Q \cap \left[\frac{1}{n+1},\frac{1}{2}\right],
		$$ 
		there exists an $n$-dimensional K-polystable (K-semistable) toric $\QQ$-Fano variety $X$ such that $\alpha(X)=r$.
	\end{thm}
	
    When $n=1$, $X$ must be $\mathbb{P}^1$, so the main theorem holds. 
    We therefore assume $n\geq 2$ in what follows.
    
	\section{Toric preliminaries}\label{sec:prelim}
	
	Let $N$ be a lattice of rank $n$ and $M=N^\vee$.  
	Let $X=X_\Sigma$ be a toric $\QQ$-Fano variety, write $v_\rho$ for the
	primitive generator of $\rho\in\Sigma(1)$ and set
	$$
	Q:=\conv\{v_\rho\mid \rho\in\Sigma(1)\}\subset N_\RR.
    $$
	Our polar convention is
	$$
	P=Q^\vee
	=\{u\in M_\RR \mid \langle u,v\rangle\geq-1,\ \forall v\in Q\}.
	$$
	Thus $P=P_{-K_X}$ is the anti-canonical polytope.  
	Its barycenter is always the volume barycenter
	\begin{equation}\label{barc}
	\barc(P)=\frac{1}{\operatorname{vol}(P)} \int_{P}u\,du,
    \end{equation}

	\begin{prop}\label{prop:toric-K-stability}
		Let $X$ be a toric $\QQ$-Fano variety. Then the following are equivalent:
		\begin{enumerate}[label=(\roman*),itemsep=6pt]
			\item $X$ is K-semistable;
			\item $X$ is K-polystable;
			\item $\barc(P)=0$.
		\end{enumerate}
	\end{prop}
	
	\begin{proof}
		The equivalence between K-semistability and 
		$\barc(P)=0$ follows from
		\cite[Corollary~7.17]{BJ20}, while the equivalence between
		K-polystability and $\barc(P)=0$ follows from
		\cite[Corollary~1.2]{Ber16}.
	\end{proof}
	
	Let
	\begin{equation}\label{eq:mu-def}
		\mu(X):=\max_{u\in P,\,v\in Q}\langle u,v\rangle.
	\end{equation}
	
	The maximum in \eqref{eq:mu-def} may equivalently be taken over the
	primitive ray generators.  
	The toric formula of Blum and Jonsson for the ordinary, non-equivariant global alpha invariant is
	\begin{equation}\label{eq:BJ-alpha}
		\alpha(X)=\frac1{1+\mu(X)};
	\end{equation}
	see \cite[Corollary~7.16]{BJ20}.
    When $X$ is a toric $\mathbb{Q}$-Fano variety, the associated polytope $P$ is rational. 
    Consequently, $\alpha(X)$ is rational.
    More generally, let $X$ be a $\mathbb{Q}$-Fano variety satisfying $\alpha(X)<1$. Then $\alpha(X)$ is rational by \cite[Theorem 1.7]{Bir21}.
	
	\begin{lem}\cite[Theorem~1.5]{LZ22}
	Let $X$ be a toric $\QQ$-Fano variety such that $\alpha(X)>\frac{1}{2}$, then $\operatorname{Aut}(X)$ is finite.
	\end{lem}
	\begin{thm}\label{thm:LZ-upper}
		Let $X$ be a toric $\QQ$-Fano variety, then 
		$$
		\alpha(X)\leq \frac{1}{2}.
		$$
	\end{thm}
	\begin{proof}
		Since $T\simeq(\mathbb C^*)^n\subset\operatorname{Aut}(X)$ is infinite, $\alpha(X)\leq 1/2$.
	\end{proof}
	
	\begin{rem}
		It also follows from this proposition that there exists no construction of examples satisfying \cite[Question~1.6(2)]{LZ22} via toric varieties.
	\end{rem}
	\begin{cor}\label{cor:necessary}
		If $X$ is an $n$-dimensional K-semistable toric $\QQ$-Fano variety,
		then
		$$
		\alpha(X)\in\QQ\cap\left[\frac1{n+1},\frac12\right].
		$$
	\end{cor}
	
	\section{Construction of Examples}\label{sec:construction}
	
	Let $r=p/q$ be a rational number with a coprime pair $p,q$.
	Let $(A,B):=(q-p,p)$, then $A,B$ are coprime and $r\in \left[1/(n+1),1/2\right]$ if and only if 
	\begin{equation}\label{eq:interval}
		B\leq A\leq nB.
	\end{equation}
	In particular, by $\gcd(A,B)=1$, $(A,B)=(1,1)$ when $A=B$ and $(A,B)=(n,1)$ when $A=nB$.
	
	Let $e_0,\dots,e_n$ be the standard coordinate vectors in $\RR^{n+1}$ and set
    $$
	M=\left\{u\in\ZZ^{n+1}\mid \sum_i u_i=0\right\},
	\qquad
	N=\ZZ^{n+1}/\ZZ(1,\dots,1).
	$$
	In $M_\RR$, we define
	\begin{equation}\label{eq:generators}
		a_i=\frac{e_i-e_{i-1}}A,
		\qquad
		b_i=\frac{e_i-e_{i+1}}B,
		\qquad 0\leq i\leq n,
	\end{equation}
	where the indices are read modulo $n+1$, and $M_\RR$ and $N_\RR$ paired by 
	$$
	\langle u,[c]\rangle:=\sum_i u_ic_i \quad \text{where}\ u\in M_\RR, [c]\in N_\RR.
	$$
	Let two polytopes be
    \begin{equation}\label{eq:PQ}
	  \begin{gathered}
		  P:=P_{n;A,B}=\conv\{a_i,b_i \mid 0\leq i\leq n\},\\[4pt]
		  Q:=P^\vee=\{[c]\in N_\RR \mid \langle u,[c] \rangle \geq -1,\ \forall u\in P\}.
	  \end{gathered}
    \end{equation}
   	
	The following propositions are the variable-width analogue of
	\cite[Proposition~2.1]{LZ26}.  
	
	\begin{prop}\label{prop:0-in-P}
		$P$ is full-dimensional in $M_\RR$ and contains $0$ in its interior.
	\end{prop}
	\begin{proof}
		Since $\{e_i\}$ spans $\RR^{n+1}$, then $\{a_i=(e_i-e_{i-1})/A\}$ spans $M_\RR$.
		Otherwise, $0$ can be written by $0=\sum_i a_i$.
		Therefore, $0$ in its interior of $\conv\{a_i \mid 0\leq i\leq n\}\subset \operatorname{Int} P$.
	\end{proof}
	
	\begin{prop}\label{prop:Q-shape}
		The cyclic difference map $N_\RR \to \RR^{n+1}$ is defined by
		$$
		[c]\longmapsto d,\quad d_i=c_{i+1}-c_i,
		$$
		induces the linear isomorphism
		$$
		N_\RR\xrightarrow{\sim}
		\left\{d\in\RR^{n+1}\mid \sum_i d_i=0\right\}.
		$$
  	    Under this identification,
	    \begin{align*}
	    	Q &= \left\{d\in N_\RR \mid \ -A\leq d_i\leq B,\ 0 \le i\le n\right\}\\ 
	    	  &= \left\{d\in\RR^{n+1} \mid \sum_i d_i=0,\ -A\leq d_i\leq B, 0 \le i\le n\right\} .
	    \end{align*}
	\end{prop}
	\begin{proof}
		Compute 
		\begin{align*}
			\left\langle a_i,[c]\right \rangle &= \langle \frac{e_i-e_{i-1}}{A},[c]\rangle\\
			                        &=\frac{c_i-c_{i-1}}{A}=\frac{d_{i-1}}{A}
		\end{align*}
		It implies that $\langle a_i,[c]\rangle\geq -1$ if and only if $d_{i-1}\geq -A$. 
		Similarly, $\langle b_i,[c]\rangle\geq -1$ if and only if $d_i\leq B$.
		Therefore, 
		   \begin{equation*}
		   	Q = \left\{d\in N_\RR \mid -A\leq d_i\leq B,\ 0 \le i\le n\right\}.  \qedhere
		   \end{equation*}
	\end{proof}
	
	\begin{prop}\label{prop:vertex-prim}
		Every vertex of $Q$ is primitive.
	\end{prop}
	\begin{proof}
		The polytope $Q$ is the intersection of the $n$-dimensional
		hyperplane $\sum d_i=0$ with the box $[-A,B]^{n+1}$.  
		Hence at a
		vertex at least $n$ coordinate bounds are active, and the coordinates cannot all be equal.
    	Let $k\geq 0$ coordinates are $-A$ and the other $n-k$ are $B$, and the free coordinate is 
    	$$
    	C=kA-(n-k)B.
    	$$ 
	    Therefore, Every vertex is of the form 
	    \begin{equation}\label{eq:vertex}
	    (\underbrace{-A,\dots,-A}_{k\text{ copies}},\,C,\,\underbrace{B,\dots,B}_{n-k\text{ copies}}),
	    \end{equation}
	    up to a permutation of the coordinates, and $A=nB$ when $k=0$.
	    Hence it's primitive by 
	    \begin{equation*}
	    	\gcd(A,B,C)=\gcd(A,B)=1.
	        \qedhere 
	    \end{equation*}	
	\end{proof}
				
    It remains to check $X$ is $\QQ$-Fano.  
    \begin{prop}\label{prop:Fano}
    	$X:=X_{\Sigma_Q}$ is $\QQ$-Fano.
    \end{prop}
    \begin{proof}
    	Let $\Sigma_Q$ be the face fan of $Q$.
    	Since every vertex of $Q$ is primitive, every facet of $Q$ provides rational Cartier data for $-K_X$. 
    	Hence $-K_X$ is $\mathbb Q$-Cartier, hence $X$ is klt.
    	Since $-K_X=\sum_v D_v$, 
    	$$
    	P_{-K_X}=\{u\in M_\RR \mid \langle u,v\rangle \geq -1,\ \forall\ \text{vertex of}\ Q\ v\}=Q^\vee=P
    	$$
    	Moreover, the fan $\Sigma_Q$ is the normal fan of the polar polytope $P=Q^\vee$. 
    	Hence the support function of $-K_X$ is strictly convex, and therefore $-K_X$ is ample.
    \end{proof}
	
	\section{Proof of Theorem~\ref{thm:main}}
	
	Finally, we will check that $\alpha(X)=r$ for given rational number $r$. 
	
	\begin{prop}
		The volume barycenter of $P$ 
		$$
		\barc(P)=0.
		$$
		Moreover, $X$ is K-polystable by Proposition~\ref{prop:toric-K-stability}. 
	\end{prop}
	\begin{proof}
		Consider the permutation $\sigma\in S_{n+1}$ is given by
		$$
		\sigma(x_0,x_1,\cdots,x_{n-1},x_n)=(x_1,x_2,\cdots,x_n,x_0).
		$$
		One has $\sigma(a_i)=a_{i-1}, \sigma(b_i)=b_{i-1}$, then $\sigma(P)=P$.
		Since $\sigma$ preserves Lebesgue measure, 
		$$
		\sigma(\barc(P))=\barc(P).
		$$
		But $\{u\in M_\RR \mid \sigma(u)=u\}=\{0\}$ implies $\barc(P)=0$.
	\end{proof}
	
	\begin{prop}
		The ordinary global alpha invariant 
		  $$
		  \alpha(X)=\frac{B}{A+B}=r
		  $$
		for the given $r$.
	\end{prop}
	\begin{proof}
		Compute
		\begin{align*}
			\mu(X)=\max_{u\in P,[c]\in Q}\langle u,[c]\rangle
			     &=\max_{u\in \{a_i,b_i\},[c]\in Q}\langle u,[c]\rangle\\
		         &=\max_{d\in Q,i}\left\{ \frac{d_{i-1}}{A},-\frac{d_i}{B}\right\}\\
			     &=\max\left\{\frac{A}{B},\frac{B}{A}\right\}=\frac{A}{B}.
		\end{align*}
		Therefore 
		\begin{equation*}
			\alpha(X)=\frac1{1+\mu(X)}=\frac{B}{A+B}=r.
			\qedhere
		\end{equation*}
	\end{proof}
	
	\begin{rem}
		For $(A,B)=(2n-1,2)$, $\alpha(X)=2/(2n+1)$; this is precisely the example of
		\cite{LZ26}.  
	\end{rem}
	
	Consequently, the main theorem Theorem~\ref{thm:main} follows.
	
	\begin{rem}
		AI assistance was used to formulate research questions, identify relevant literature on recent methods, suggest candidate examples, check properties of those examples, prepare the initial \TeX{} source, and conduct a preliminary review. 
		The author independently made the final selection of examples, carried out all mathematical refinements and rigorous arguments, and completed the final revisions. 
		All AI-assisted mathematical checks were independently verified by the author before being incorporated into the manuscript.
	\end{rem}
	
	\section{Properties of $X_{n;A,B}$}
	
	Write 
	$$
	(n+1)B=k(A+B)+m,\ \text{where}\ 0 \leq m<A+B,\ k,m\in \mathbb{Z}.
	$$
	\begin{lem}The vertices of Q are as follows.
		 Let $R:=\# \operatorname{Vert}(Q)$ be the number of vertices of $Q$.
		 When
		\begin{enumerate}[label=(\roman*),itemsep=6pt]
			\item $m=0$. 
			The vertices of $Q$: 
	        $$
            v=(\underbrace{-A,\dots,-A}_{k\text{ copies}},\,\,\underbrace{B,\dots,B}_{n+1-k\text{ copies}}),
            $$ 
			and its all permutations of the coordinates.
			In this case, $R=\dbinom{n+1}{k}$.
			
			\item $m>0$.  
			The vertices of $Q$: 
	        $$
            v=(\underbrace{-A,\dots,-A}_{k\text{ copies}},\,C,\,\underbrace{B,\dots,B}_{n-k\text{ copies}}),
            $$ 
			and its all permutations of the coordinates.
			In this case, $R=(n+1)\dbinom{n}{k}$.
		\end{enumerate}
	\end{lem}
		
	\begin{prop}
		The following properties hold for such $X_{n;A,B}$.
		\begin{enumerate}[label=(\roman*),itemsep=6pt]
			\item $X_{n;A,B}$ is Gorenstein when $(A,B)=(n,1)$.
			Moreover, the index of $X$ is $AB$ when $A<nB$;
			
			\item $X_{n;A,B}$ is $\mathbb{Q}$-factorial if and only if 
			  \begin{enumerate}
				  \item $n=2$ or
				  \item $(n;A,B)=(n;n,1)$ or $(3;1,1)$;
			  \end{enumerate}
			  
			\item For $n\ge 2$, the simplest example is
			$$
			X_{2,1,1}\simeq \operatorname{Bl}_{\{3\text{pts}\}} \mathbb{P}^2$$
			which is the blow-up of $\mathbb{P}^2$ at three torus-invariant points. It is the only smooth (resp., terminal) example among $X_{n;A,B}$. Moreover, $X_{n;A,B}$ is canonical if and only if $B=1$;
			
			\item Let $g=\gcd(n+1,A+B)$, then
			$$ 
			\operatorname{Cl}(X_{n;A,B}) \simeq \ZZ^{R-n}\oplus (\ZZ/g\ZZ)^{n-1};
			$$
			
			\item The volume of $-K_{X_{n;A,B}}$ is
			$$
			\operatorname{vol}(-K_X)=(-K_X)^n=n!\operatorname{vol}_M(P)=\frac{(n+1)\binom{n}{k}}{A^kB^{n-k}}.
			$$
		\end{enumerate}
		We omit the proofs, since all of them amount to standard exercises. 
		We record these results only so that readers may readily reference these properties.
	\end{prop}
	
	\section*{Acknowledgements}
       I am grateful to Professor Yuchen Liu for carefully reading the manuscript and for his many insightful comments and suggestions, which greatly improved this work. 
       I also thank him for proposing a natural follow-up problem motivated by the examples constructed in this paper.
       I would also like to express my gratitude to my supervisor, Professor Zhixian Zhu, and Professor Lei Song for their support, with particular thanks to Professor Zhu for advising me to include an explicit disclosure of the role of AI assistance in this work.
	
   \bibliographystyle{amsalpha}
   \bibliography{reference}

\providecommand{\bysame}{\leavevmode\hbox to3em{\hrulefill}\thinspace}
\providecommand{\MR}{\relax\ifhmode\unskip\space\fi MR }
\providecommand{\MRhref}[2]{%
  \href{http://www.ams.org/mathscinet-getitem?mr=#1}{#2}
}
\providecommand{\href}[2]{#2}
\begin{thebibliography}{Ber16}

\bibitem[Ber16]{Ber16}
Robert~J. Berman, \emph{{K-polystability of Q-Fano varieties admitting
  K{\"a}hler--Einstein metrics}}, Inventiones Mathematicae \textbf{203} (2016),
  no.~3, 973--1025.

\bibitem[Bir21]{Bir21}
Caucher Birkar, \emph{{Singularities of linear systems and boundedness of Fano
  varieties}}, Annals of Mathematics \textbf{193} (2021), no.~2, 347--405.

\bibitem[BJ20]{BJ20}
Harold Blum and Mattias Jonsson, \emph{Thresholds, valuations, and
  k-stability}, Advances in Mathematics \textbf{365} (2020), 107062.

\bibitem[FO18]{FO18}
Kento Fujita and Yuji Odaka, \emph{{On the K-stability of Fano varieties and
  anticanonical divisors}}, Tohoku Mathematical Journal. Second Series
  \textbf{70} (2018), no.~4, 511--521.

\bibitem[Jia17]{Jia17}
Chen Jiang, \emph{{K-semistable Fano manifolds with the smallest alpha
  invariant}}, International Journal of Mathematics \textbf{28} (2017), no.~6,
  1750044.

\bibitem[LZ22]{LZ22}
Yuchen Liu and Ziquan Zhuang, \emph{{On the sharpness of Tian's criterion for
  K-stability}}, Nagoya Mathematical Journal \textbf{245} (2022), 41--73.

\bibitem[LZ26]{LZ26}
Jihao Liu and Ziwen Zhu, \emph{{K-polystable toric Fano varieties with small
  alpha invariants}}, July 2026.

\end{thebibliography}
	
\end{document}